\documentclass[12pt,a4paper]{amsart}
\usepackage[a4paper,left=2.35cm,right=2.35cm,top=3cm,bottom=3cm,headsep=1cm]{geometry}

\usepackage{amssymb, stmaryrd, blkarray}
\usepackage{tikz-cd}
\usepackage{tikz}\usetikzlibrary{graphs,quotes,fit,positioning,matrix,calc,decorations.markings,angles,decorations.pathmorphing,decorations.pathreplacing}
\usepackage{mathtools}
\usepackage[shortlabels]{enumitem}
\usepackage{graphicx}
\usepackage{picins}
\usepackage{euler}

\newcommand\sss{\scriptstyle}

\newcommand{\nc}{\newcommand}
\nc{\la}{\lambda}
\nc{\ep}{\epsilon}
\nc{\C}{\mathbb{C}}
\nc{\Z}{\mathbb{Z}}
\nc{\content}{\operatorname{content}}
\nc{\Span}{\operatorname{Span}}
\newcommand\hp{{\tikz[baseline=0,scale=0.3]{\halfuptri{}{}}}}
\nc{\Id}{\operatorname{Id}}
\nc{\fug}{\mathit{fug}}
\nc{\wt}{\widetilde}

\usepackage{tikz}\usetikzlibrary{graphs,quotes,fit,positioning,matrix,calc,decorations.markings,angles,decorations.pathmorphing,decorations.pathreplacing}

\tikzset{mynode/.style={circle,draw=black,fill=black,inner sep=1.8pt,outer sep=0pt}}
\tikzset{edgelabel/.style={\mcol,inner sep=0pt}}
\tikzset{invlabel/.style={draw=black,text=black,circle,inner sep=0pt,minimum size=3mm}}
\newcommand\tikzif[2][]{
	\tikzifinpicture{#2}{\begin{tikzpicture}[#1]#2\end{tikzpicture}}
}
\tikzset{math mode/.style = {execute at begin node=$, execute at end node=$}}
\def\dr{red!80!black} 
\def\dg{green!80!black} 
\def\db{blue!80!black} 
\tikzset{d/.style={ultra thick}}
\tikzset{dr/.style={draw=\dr,d}}
\tikzset{dg/.style={draw=\dg,d}}
\tikzset{db/.style={draw=\db,d}} 
\def\mcol{black}
\def\m#1{{\color{\mcol}#1}}
\tikzset{rt/.style={text=blue,execute at begin node=$\sf,execute at end node=$}}

\tikzset{arrow/.style={postaction={decorate,thick,decoration={markings,mark = at position #1 with {\arrow{>}}}}},arrow/.default=0.5}
\tikzset{invarrow/.style={postaction={decorate,thick,decoration={markings,mark = at position #1 with {\arrow{<}}}}},invarrow/.default=0.5}
\long\def\remint#1#2{%
	\tikz[baseline=-4pt]%
	{\node[outer sep=0pt,draw=black,fill=cyan!40!green!50!white,rectangle,rounded corners,align=left,text width=#1]{#2}}%
}
\newbox{\rembox}
\long\def\icom#1{\noindent\nobreak\hfil\penalty1000\hfilneg
	\sbox{\rembox}{#1}%
	\ifdim\wd\rembox>\textwidth\remint{\textwidth}{#1}\else\remint{}{#1}\fi%
}
\newlength\myshift
\newcommand\getshift{%
	\pgfmathsetlength{\myshift}{0.8mm}
}
\def\noparenx#1{%
	\ifx\relax#1
	\else
	\if)#1%
	\else
	\if(#1%
	\else
	#1%
	\fi
	\fi
	\expandafter\noparenx
	\fi
}
\def\noparen#1{\noparenx #1\relax}
\newcommand\rh[5][]{
	\tikzif[baseline=0,scale=1.25]{
		\getshift
		\draw[thick,black] (0,0) coordinate (ad) -- node[edgelabel,left,xshift=\myshift] {$#2$} ++(-60:1) coordinate (ba) -- node[edgelabel,right,xshift=-\myshift] {$#3$} ++(60:1) coordinate (cb) -- node[edgelabel,right,xshift=-\myshift] {$#4$} ++(120:1) coordinate (dc) -- node[edgelabel,left,xshift=\myshift] {$#5$} cycle;
		\ifx\&#1\&\else\node[invlabel] at (0.5,0) {$\sss#1$};\fi
}}

\newcommand\uptri[4][]{
	\tikzif[baseline=0.34cm,scale=1.25]{
		\getshift
		\draw[thick,black] (-0.5,0) -- node[edgelabel] (horiz) {$\vphantom{\noparen{#3}}\smash{#3}$} ++(0:1) -- node[edgelabel,right,xshift=-\myshift] (NE) {$#4$} ++(120:1) -- node[edgelabel,left,xshift=\myshift] (NW) {$#2$} ++(240:1) -- cycle; 
		\ifx\&#1\&\else\node[invlabel] at (0,0.33) {$\sss#1$};\fi
}}
\newcommand\downtri[4][]{
	\tikzif[baseline=-0.54cm,scale=1.25]{
		\getshift
		\begin{scope}[scale=-1]
			\draw[thick,black] (-0.5,0) -- node[edgelabel] (horiz) {$\vphantom{\noparen{#3}}\smash{#3}$} ++(0:1) -- node[edgelabel,left,xshift=\myshift] (NE) {$#4$} ++(120:1) -- node[edgelabel,right,xshift=-\myshift] (NW) {$#2$} ++(240:1) -- cycle; 
			\ifx\&#1\&\else\node[invlabel] at (0,0.33) {$\sss#1$};\fi
		\end{scope}
}}
\newcommand\halfuptri[3][]{
	\tikzif[baseline=0.34cm,scale=1.25]{
		\getshift
		\draw[thick,black] (-0.5,0) ++(60:1) coordinate (A) -- node[edgelabel,left,xshift=\myshift] (NW) {$#2$} ++(240:1) -- node[edgelabel] (horiz) {$\vphantom{\noparen{#3}}\smash{#3}$} ++(0:0.5) coordinate (B); \draw[thick,black,dash pattern=on 1.5pt off 1.5pt] (A) -- (B);
		\ifx\&#1\&\else\node[invlabel] at (-0.15,0.25) {$\sss#1$};\fi
}}

\theoremstyle{plain}
\newtheorem{thm}{Theorem}

\newtheorem{lem}[thm]{Lemma}
\newtheorem{prop}[thm]{Proposition}
\newtheorem{defn}[thm]{Definition}
\newtheorem{rem}[thm]{Remark}
\newtheorem{eg}[thm]{Example}

\newtheorem{innercustomthm}{Theorem}
\newenvironment{customthm}[1]
{\renewcommand\theinnercustomthm{#1}\innercustomthm}
{\endinnercustomthm}

\newcommand\dom\backslash

\title[Restricting Schubert classes to symplectic Grassmannians using self-dual puzzles]
{Restricting Schubert classes to \\
	symplectic Grassmannians using self-dual puzzles}

\title[Restricting Schubert classes to symplectic Grassmannians using puzzles]
{Restricting Schubert classes to \\
	symplectic Grassmannians using self-dual puzzles}
\author{Iva Halacheva}
\author{Allen Knutson}
\author{Paul Zinn-Justin}
\address{Iva Halacheva, Paul Zinn-Justin, School of Mathematics and Statistics, The University of Melbourne, 
	Victoria 3010, Australia}
\email{iva.halacheva@unimelb.edu.au}
\email{pzinn@unimelb.edu.au}
\address{Allen Knutson, Cornell University, Ithaca, New York}
\email{allenk@math.cornell.edu}
\thanks{PZJ was supported by ARC grant FT150100232.
}
\date{\today}

\begin{document}
\begin{abstract}
	Given a Schubert class on $Gr(k,V)$ where $V$ is a
	symplectic vector space of dimension $2n$, we consider its restriction to the
	symplectic Grassmannian $SpGr(k,V)$ of isotropic subspaces.
	Pragacz gave tableau formul\ae\ for positively computing the expansion
	of these $H^*(Gr(k,V))$ classes into Schubert classes of the target
	when $k=n$, which corresponds to expanding Schur polynomials into
	$Q$-Schur polynomials.
	Co\c skun described an algorithm for their expansion when $k\leq n$.
	We give a puzzle-based formula for these expansions, while
	extending them to equivariant cohomology. We make use of a new
	observation that usual Grassmannian puzzle pieces are already enough
	to do some $2$-step Schubert calculus, and apply techniques from quantum integrable systems (``scattering diagrams'').
\end{abstract}
	
	\maketitle
	
	\newcommand\junk[1]{}

	
	\section{Introduction}
	
	\newcommand\bfdefn[1]{{\bf #1}}
	
	\newcommand\iso\cong
	\newcommand\into\hookrightarrow
	\newcommand\onto\twoheadrightarrow
	
	\subsection{Grassmannian duality of puzzles}\label{ssec:duality}
	
	The Littlewood--Richardson coefficients $c_{\lambda\mu}^\nu$,
	where $\lambda,\mu,\nu$ are (for now) partitions, satisfy
	a number of symmetries, one of which is 
	$c_{\lambda\mu}^\nu = c_{\mu^T \lambda^T}^{\nu^T} $.
	One origin of L-R coefficients is as structure constants in
	the product in $H^*(Gr(k,V))$ of Schubert classes on the Grassmannian
	of $k$-planes in $V$.
	In that formulation, the \bfdefn{Grassmannian duality} homeomorphism 
	$Gr(k,V) \iso Gr((\dim V)-k,V^*)$, $(U \leq V) \mapsto (U^\perp \leq V^*)$ 
	induces an isomorphism of cohomology rings
	and a correspondence of Schubert bases, giving the symmetry above.
	This symmetry is not at all manifest in tableau-based computations
	of the $\{c_{\lambda\mu}^\nu\}$, but it is in the ``puzzle'' rule
	of \cite{KT03}, which replaces partitions by binary strings and is based on the puzzle pieces
	$$
	\uptri{0}{0}{0}\quad \uptri{1}{1}{1}\quad \uptri{1}{10}{0}\ 
	\text{, their rotations, and the equivariant piece }  \rh{1}{0}{1}{0}; 
	$$
	we recall and generalize this rule in Theorem \ref{thm:jk} below.
	Specifically, the \bfdefn{dual} of a puzzle is made 
	by flipping
	it left-right while 
	exchanging all $0 \leftrightarrow 1$
	(in particular, $10$-labels again become $10$s). The duals of the
	puzzles counted by $c_{\lambda\mu}^\nu$ are exactly those counted by
	$c_{\mu^T\lambda^T}^{\nu^T}$.
	
	This prompts the question: what do {\em self-dual} puzzles count?  
	One might expect it is something related to an isomorphism $V\iso V^*$
	i.e. a bilinear form, and indeed our main theorems \ref{thm:n2n},
	\ref{thm:generalk}, and \ref{thm:main} interpret self-dual puzzles as computing the
	restrictions of Schubert classes on $Gr(k,2n)$ to the {\em symplectic Grassmannian} $SpGr(k,2n)$.
	(We will address elsewhere the minimal modifications necessary to handle
	the orthogonal case.)
	For $k=n$, there was already a tableau-based formula for these
	restrictions\footnote{In particular, \cite{Pragacz} provides a cohomological 
		interpretation of algebraic results of Stembridge \cite{Stembridge}
		about expanding Schur functions into 
		Schur $P$- and $Q$-functions.}
	in \cite{Pragacz} which is less simple to state than 
	Theorem \ref{thm:n2n}; see also \cite{CoskunSp}. 
	This is perhaps another effect of tableaux
	being less suited to Grassmannian duality than puzzles are.  
	
	\subsection{Restriction from $Gr(n,2n)$}
	
	
	Let $V$ be a vector space over $\C$ equipped with a symplectic form,
	so the Grassmannian $Gr(k,V)$ of $k$-planes contains the subscheme
	$$ SpGr(k,V) := \{L \leq V\ :\ \dim L = k,\ L \leq L^\perp \} $$
	where $\perp$ means perpendicular with respect to the symplectic form. 
	Then the inclusion $\iota:\ SpGr(k,V) \into Gr(k,V)$ induces a pullback 
	$ \iota^*:\ H^*(Gr(k,V)) \to H^*(SpGr(k,V)) $
	in cohomology. As both cohomology rings possess bases consisting of
	{\em Schubert classes} $\{S_\lambda\}$, one can ask about expanding 
	$\iota^*(S_\lambda)$ in the basis of $SpGr(k,V)$'s Schubert classes $\{S_\nu\}$.

	Let $\dim V = 2n$ (necessarily even, since $V$ is symplectic),
	and for the simplest version of the theorem assume $k=n$. 
	Then the Schubert classes on $Gr(n,V)$ are indexed by the
	$2n\choose n$ binary strings with $n$ $0$s and $n$ $1$s, whereas the
	Schubert classes on $SpGr(n,V)$ are indexed by the $2^n$ binary strings 
	of length $n$ (with more detail on this indexing in \S \ref{sec:setup}).

	\begin{customthm}{1A}\label{thm:n2n}
		Let $S_\lambda$ be a Schubert class on $Gr(n,2n)$, indexed by a string 
		$\lambda$ with content in $0^n 1^n$, and $S_\nu$ a Schubert class on 
		$SpGr(n,2n)$, indexed by a length $n$ binary string. Then the coefficient
		of $S_\nu$ in $\iota^*(S_\lambda)$ is the number of self-dual puzzles
		with $\lambda$ on the Northwest side, $\nu$ on the left half of the 
		South side (both $\lambda$ and $\nu$ read left to right), and
		equivariant pieces only allowed along the axis of reflection.
	\end{customthm}

	\parpic[r][b]{%
		\begin{minipage}[b]{53mm}\vspace{-0.3cm}
			\[\begin{tikzpicture}[math mode,nodes={\mcol},x={(-0.577cm,-1cm)},y={(0.577cm,-1cm)},scale=1]
			\draw[thick] (0,0)
			-- node {\sss 0} ++(0,1); \draw[thick] (0,0)
			-- node {\sss 1} ++(1,0); \draw[thick] (0+1,0)
			-- node {\sss 10} ++(-1,1); 
			\draw[thick] (0,1)
			-- node {\sss 1} ++(0,1); \draw[thick] (0,1)
			-- node {\sss 1} ++(1,0); \draw[thick] (0+1,1)
			-- node {\sss 1} ++(-1,1); 
			\draw[thick] (0,2)
			-- node {\sss 0} ++(0,1); \draw[thick] (0,2)
			-- node {\sss 1} ++(1,0); \draw[thick] (0+1,2)
			-- node {\sss 10} ++(-1,1); 
			\draw[thick] (0,3)
			-- node {\sss 1} ++(0,1); \draw[thick] (0,3)
			-- node {\sss 1} ++(1,0); \draw[thick] (0+1,3)
			-- node {\sss 1} ++(-1,1); 
			\draw[thick] (1,0)
			-- node {\sss 0} ++(0,1); \draw[thick] (1,0)
			-- node {\sss 0} ++(1,0); \draw[thick] (1+1,0)
			-- node {\sss 0} ++(-1,1); 
			\draw[thick] (1,1)
			-- node {\sss 1} ++(0,1); \draw[thick] (1,1)
			-- node {\sss 0} ++(1,0); \draw[thick] (1+1,1);
			\draw[thick] (1,2)
			-- node {\sss 0} ++(0,1); \draw[thick] (1,2)
			-- node {\sss 0} ++(1,0); \draw[thick] (1+1,2)
			-- node {\sss 0} ++(-1,1); 
			\draw[thick] (2,0)
			-- node {\sss 0} ++(0,1); \draw[thick] (2,0)
			-- node {\sss 1} ++(1,0); \draw[thick] (2+1,0)
			-- node {\sss 10} ++(-1,1); 
			\draw[thick] (2,1)
			-- node {\sss 1} ++(0,1); \draw[thick] (2,1)
			-- node {\sss 1} ++(1,0); \draw[thick] (2+1,1)
			-- node {\sss 1} ++(-1,1); 
			\draw[thick] (3,0)
			-- node {\sss 0} ++(0,1); \draw[thick] (3,0)
			-- node {\sss 0} ++(1,0); \draw[thick] (3+1,0)
			-- node {\sss 0} ++(-1,1); 
			\end{tikzpicture}\]
		\end{minipage}
	}
	\begin{eg} 
		For $\la=0101$, a self-dual puzzle with $\la$ on the Northwest
		side has to be of the form \scalebox{0.7}{$\uptri{\la}{\mu}{\la}$}
		for some $\mu$.
		
		So, it will appear in the usual calculation of
		$S_{0101}^2 \in H^*_T(Gr(2,4))$, which involves three
		puzzles. Only one of these puzzles is self-dual, and its only
		equivariant piece is on the centerline. From this we compute
		$\iota^*(S_{0101}) = S_{01}$ in $H^*(SpGr(2,4))$.
	\end{eg}
	
	\vspace{0.2cm} 
	
	A surprising aspect of Theorem \ref{thm:n2n} is that equivariant pieces appear in this nonequivariant calculation, albeit only down the centerline.\footnote{The number of equivariant pieces down the centerline is in fact fixed and equal
		to the number of $1$s in $\nu$, by a weight conservation argument.} 
	If we allow them elsewhere (self-dually occurring in pairs), 
	the puzzles compute the generalization of Theorem \ref{thm:n2n} to the map 
	$\iota^*:\ H^*_T(Gr(n,2n)) \to H^*_T(SpGr(n,2n))$ 
	in {\em (torus-)equivariant} cohomology, whose coefficients now live in the
	polynomial ring $H^*_T(pt) \cong {\mathbb Z}[y_1,\ldots,y_n]$.
	We leave this statement until Theorem \ref{thm:main} in \S \ref{sec:proof} because it requires some
	precision about the locations of the symplectic Schubert varieties.
	
	\subsection{Interlude: puzzles with $10$s on the South side}
	\label{ssec:south10s}
	
	To generalize Theorem \ref{thm:n2n} to $SpGr(k,2n)$, not just $k=n$,
	we need strings that index its 
	${n \choose k} 2^k$ many Schubert classes. We do this using the third 
	edge label, $10$: consider strings $\nu$ of length $n$ with $(n-k)$ $10$s,
	the rest a mix of $1$s and $0$s. 
	
	Before considering {\em self-dual} puzzles with Southside $10$s, we
	mention a heretofore unobserved capacity of the puzzle pieces from
	\cite{KT03}, available once we allow for Southside $10$s.
	It turns out they are already
	sufficient to compute certain products\footnote{Note that the general
		$2$-step problem has received a puzzle formula \cite{BKPT}, but
		using many more puzzle pieces than we use here. The problem of
		multiplying classes from different Grassmannians was studied already
		in \cite{PurbhooSottile}.}  in the $T$-equivariant cohomology of
	$2$-step flag manifolds! The only necessary new idea is to allow the
	previously internal label $10$ to appear on the South side.
	
	\begin{thm}\label{thm:jk}
		Let $0 \leq j \leq k \leq n$, and let $\lambda$, $\mu$ be $0,1$-strings
		with content $0^j 1^{n-j}$, $0^k 1^{n-k}$ respectively, 
		defining equivariant Schubert classes $S_\lambda,S_\mu$ on $Gr(j,\C^n)$, $Gr(k,\C^n)$ respectively.
		Let $\pi_j,\pi_k$ be the respective projections of the $2$-step flag manifold 
		$Fl(j,k;\ \C^n)$ to those Grassmannians. 
		Let $\nu$ be a string in the ordered alphabet $0, 10, 1$ with content
		$0^j (10)^{k-j} 1^{n-k}$, defining a Schubert class $S_\nu$ in
		$H^*_T(Fl(j,k;\ \C^n))$. 
		We emphasize that the alphabet order is $0$, $10$, $1$!
		
		Then as in \cite{KT03}, the coefficient of $S_\nu$ in the product
		$\pi_i^*(S_\lambda) \pi_j^*(S_\mu) \in H^*_T(Fl(j,k;\ \C^n))$
		is the sum over puzzles $P$ with boundary labels
		$\lambda$, $\mu$, $\nu$, made from the puzzle pieces in \S \ref{ssec:duality},
		of the ``fugacities'' 
		$\fug(P) := \prod_{\text{equivariant pieces $\diamondsuit$ in $P$}}
		(y_{\text{NE--SW diagonal of }\diamondsuit} - y_{\text{NW--SE diagonal of }\diamondsuit})$.
	\end{thm}
	
	
	\parpic[r][b]{%
		\begin{minipage}[b]{65mm}\vspace{0.2cm}
			$ 
			\begin{matrix}
			\ 
			\begin{tikzpicture}[math mode,nodes={\mcol},x={(-0.577cm,-1cm)},y={(0.577cm,-1cm)},scale=.7]
			\draw[thick] (0,0)
			-- node[pos=0.5] {\sss 1} ++(0,1); \draw[thick] (0,0)
			-- node[pos=0.5] {\sss 1} ++(1,0); \draw[thick] (0+1,0)
			-- node          {\sss 1} ++(-1,1); 
			\draw[thick] (0,1)
			-- node[pos=0.5] {\sss 0} ++(0,1); \draw[thick] (0,1)
			-- node[pos=0.5] {\sss 1} ++(1,0); \draw[thick] (0+1,1)
			-- node          {\sss 10} ++(-1,1); 
			\draw[thick] (0,2)
			-- node[pos=0.5] {\sss 0} ++(0,1); \draw[thick] (0,2)
			-- node[pos=0.5] {\sss 1} ++(1,0); \draw[thick] (0+1,2)
			-- node          {\sss 10} ++(-1,1); 
			\draw[thick] (1,0)
			-- node[pos=0.5] {\sss 1} ++(0,1); \draw[thick] (1,0)
			-- node[pos=0.5] {\sss 0} ++(1,0); \draw[thick] (1+1,0)
			; 
			\draw[thick] (1,1)
			-- node[pos=0.5] {\sss 0} ++(0,1); \draw[thick] (1,1)
			-- node[pos=0.5] {\sss 0} ++(1,0); \draw[thick] (1+1,1)
			-- node          {\sss 0} ++(-1,1); 
			\draw[thick] (2,0)
			-- node[pos=0.5] {\sss 1} ++(0,1); \draw[thick] (2,0)
			-- node[pos=0.5] {\sss 1} ++(1,0); \draw[thick] (2+1,0)
			-- node          {\sss 1} ++(-1,1); 
			\end{tikzpicture}
			& \quad &
			\begin{tikzpicture}[math mode,nodes={\mcol},x={(-0.577cm,-1cm)},y={(0.577cm,-1cm)},scale=.7]
			\draw[thick] (0,0)
			-- node[pos=0.5] {\sss 1} ++(0,1); \draw[thick] (0,0)
			-- node[pos=0.5] {\sss 1} ++(1,0); \draw[thick] (0+1,0)
			-- node          {\sss 1} ++(-1,1); 
			\draw[thick] (0,1)
			-- node[pos=0.5] {\sss 0} ++(0,1); \draw[thick] (0,1)
			-- node[pos=0.5] {\sss 0} ++(1,0); \draw[thick] (0+1,1)
			-- node          {\sss 0} ++(-1,1); 
			\draw[thick] (0,2)
			-- node[pos=0.5] {\sss 0} ++(0,1); \draw[thick] (0,2)
			-- node[pos=0.5] {\sss 0} ++(1,0); \draw[thick] (0+1,2)
			-- node          {\sss 0} ++(-1,1); 
			\draw[thick] (1,0)
			-- node[pos=0.5] {\sss 10} ++(0,1); \draw[thick] (1,0)
			-- node[pos=0.5] {\sss 0} ++(1,0); \draw[thick] (1+1,0)
			-- node          {\sss 1} ++(-1,1); 
			\draw[thick] (1,1)
			-- node[pos=0.5] {\sss 0} ++(0,1); \draw[thick] (1,1)
			-- node[pos=0.5] {\sss 1} ++(1,0); \draw[thick] (1+1,1)
			-- node          {\sss 10} ++(-1,1); 
			\draw[thick] (2,0)
			-- node[pos=0.5] {\sss 1} ++(0,1); \draw[thick] (2,0)
			-- node[pos=0.5] {\sss 1} ++(1,0); \draw[thick] (2+1,0)
			-- node          {\sss 1} ++(-1,1); 
			\end{tikzpicture}
			\end{matrix}
			$
		\end{minipage}
	}

	\begin{eg} 
		If $\lambda = 101$, $\mu = 100$, then their pullbacks give
		$\pi_1^*(S_{101}) = S_{10,0,1}$, $\pi_2^*(S_{100}) = S_{1,0,10}$, with product
		$(y_1-y_2)S_{1,0,10} + S_{1,10,0}$ (note: to compare strings to
		permutations requires inversion, as in \S \ref{sec:AJSBilley}).
	\end{eg}

	\subsection{Restriction from $Gr(k,2n)$, $k<n$}

	\begin{customthm}{1B}\label{thm:generalk}
		Let $\lambda$ be a string with content $0^k 1^{2n-k}$, whereas $\nu$ is 
		of length $n$ with $(n-k)$ $10$s, the rest a mix of $1$s and $0$s. 
		Consider the puzzles from Theorem \ref{thm:jk}, where 
		we allow $10$ labels to appear on the South side.
		
		Then as before, in $H^{\ast}(SpGr(k,2n))$, the coefficient of $S_\nu$ in $\iota^*(S_\lambda)$
		is the number of self-dual puzzles with $\lambda$ on the Northwest side, 
		$\nu$ on the left half of the South side (both $\lambda$ and $\nu$
		read left to right), and equivariant pieces only allowed along the
		axis of reflection.  
	\end{customthm}
	
	\newcommand\fulltri[5]{
		\draw[thick] (#1,#2) 	-- node {\sss #3} ++(1,0); 
		\draw[thick] (#1,#2) 	-- node {\sss #4} ++(0,1); 
		\draw[thick] (#1+1,#2) 	-- node {\sss #5} ++(-1,1); 
	}
	\newcommand\andNE[6]{
		\fulltri{#1}{#2}{#3}{#4}{#5}
		\draw[thick] (#1,#2) 	-- node {\sss #6} ++(-1,0); 
	}
	\newcommand\eqvttop[2]{
		\draw[thick] (#1,#2) 	-- node {\sss 0} ++(1,0); 
		\draw[thick] (#1,#2) 	-- node {\sss 1} ++(0,1); 
	}

	\parpic[r][b]{%
		\begin{minipage}[b]{75mm}\vspace{-0.2cm}
			$
			\begin{tikzpicture} 
			[math mode,nodes={\mcol},x={(-0.577cm,-1cm)},y={(0.577cm,-1cm)},scale=.6]
			\andNE   1 0 0 0 0 1
			\fulltri 2 0 1 0 {10}
			\fulltri 3 0 0 0 0
			\andNE   2 1 1 1 1 0
			\fulltri 4 0 1 0 {10}
			\eqvttop 3 1
			\fulltri 5 0 1 0 {10}
			\fulltri 4 1 1 1 1
			\andNE   3 2 0 0 0 1
			\draw[dashed,d] (0,0) -- (3,3);
			\end{tikzpicture}
			\quad
			\begin{tikzpicture} 
			[math mode,nodes={\mcol},x={(-0.577cm,-1cm)},y={(0.577cm,-1cm)},scale=.6]
			\andNE   1 0 0 0 0 1
			\fulltri 2 0 1 0 {10}
			\fulltri 3 0 0 0 0
			\andNE   2 1 1 1 1 0
			\fulltri 4 0 1 0 {10}
			\fulltri 5 0 1 0 {10}
			\fulltri 4 1 1 1 1
			\fulltri 3 1 0 {10} 1
			\andNE   3 2 1 1 1 0
			\draw[dashed,d] (0,0) -- (3,3);
			\end{tikzpicture}
			\quad
			\begin{tikzpicture} 
			[math mode,nodes={\mcol},x={(-0.577cm,-1cm)},y={(0.577cm,-1cm)},scale=.6]
			\andNE   1 0 0 0 0 1
			\fulltri 2 0 1 0 {10}
			\fulltri 3 0 0 0 0
			\andNE   2 1 1 1 1 0
			\fulltri 4 0 1 1 1
			\fulltri 5 0 1 1 1
			\fulltri 3 1 {10} 1 0
			\fulltri 4 1 1 0 {10}
			\andNE   3 2 0 0 0 1
			\draw[dashed,very thin,d] (0,0) -- (3,3);
			\end{tikzpicture}
			$
		\end{minipage}
	}
	\begin{eg}\vspace{0.2cm}
		In the remainder of the paper we work with the left halves $\hp$
		of self-dual puzzles, since the centerline and right half can be inferred.
		The half-puzzles pictured here (really for equivariant Theorem
		\ref{thm:main} to come) show
		$\iota^*(S_{110101}) = (y_2-y_3)S_{10,1,0} + S_{10,1,1} + S_{1,10,0}$.\vspace{0.3cm}
	\end{eg}
	
	\vspace{-0.1cm}
	
	\junk{
		In this extended abstract we only include the proof of Theorem
		\ref{thm:generalk} (containing Theorem \ref{thm:n2n})
		in ordinary and equivariant cohomology (see Theorem \ref{thm:main}); the $K$-version will appear elsewhere.
	}

	The proof is based on the ``quantum integrability'' of $R$-matrices, 
	and closely follows that of \cite{KZJ} (see also \cite{ZJ}); in particular, following the quantum integrable literature,
	we use graph-dual pictures (scattering diagrams) which are more amenable 
	than puzzles to topological manipulations.
	The principal new feature is
	the appearance of $K$-matrices. The ``reflection equation''
	$RKRK = KRKR$ (more precisely, eq. (\ref{eqn:reflection})
	in Lemma \ref{lem:TCP}) is standard; however since the approach
	of \cite{KZJ} requires not just $R$-matrices but the trivalent $U$-matrix, 
	we need here the possibly novel ``$K$-fusion equation'' (\ref{eqn:fusion})
	in Lemma \ref{lem:TCP}.

	\section{The groups, flag manifolds, and
		cohomology rings}\label{sec:setup}
	
	\junk{
		While all symplectic structures on $\C^{2n}$ are
		$GL_{2n}$-equivalent, in order to compute in $T^n$-equivariant
		cohomology we need to be more specific about our choice of form: }
	We take the Gram matrix of our symplectic form
	to be antidiagonal; this is so that if $B_\pm$
	are the upper/lower triangular Borel subgroups of $GL_{2n}$, then
	$B_\pm \cap Sp_{2n}$ will be opposed Borel subgroups of $Sp_{2n}$.
	
	Consider
	$Gr(k,2n) = \{0 \leq V \leq \C^{2n} \; | \; \dim{V}=k\} \iso GL_{2n}/P$
	where $k\leq n$ and $P$ is the parabolic subgroup of block type
	$(k,2n-k)$ containing $B = B_+$.
	Then $ P_{Sp_{2n}}=P \cap Sp_{2n}$ is a parabolic for the
	Lie subgroup $Sp_{2n}$ and the symplectic Grassmannian is
	$SpGr(k,2n)=\{0 \leq V < \C^{2n} \; | \; \dim{V}=k,
	V\leq V^{\perp}\} \iso Sp_{2n}/P_{ Sp_{2n}}$. 
	Let $T^{2n} := B_+ \cap B_-$ be the diagonal matrices in $GL_{2n}$, 
	and $T^{n} := Sp_{2n} \cap T^{2n}$.
	Note that $(GL_{2n}/P)^{T^n} = (GL_{2n}/P)^{T^{2n}}$ since there exist
	$x \in T^n$ with no repeated eigenvalues. The following diagram of spaces commutes.
	
	$$
	\begin{tikzcd}
	\{ \nu\ :\ \exists \; a \text{ s.t. } \content(\nu) = 0^a (10)^{n-k} 1^{k-a} \}
	\arrow[hookrightarrow]{d}{\wt{{\iota}}}
	\arrow["\sim"]{r} & 
	(SpGr(k,2n))^{T^n} \arrow[hookrightarrow]{r}
	\arrow[hookrightarrow]{d} & 
	SpGr(k,2n) 
	\arrow[hookrightarrow]{d}{\iota}  
	\\   
	\{ \lambda\ :\ \content(\lambda) = 0^k 1^{2n-k} \}
	\arrow["coord"]{r} & 
	(Gr(k,2n))^{T^n} 
	\arrow[hookrightarrow]{r}
	&
	Gr(k,2n)  
	\end{tikzcd}
	$$
	The map $\wt{{\iota}}$ takes a sequence $\nu$ first to its double
	$\nu \overline{\nu}$ where $\overline{\nu}$ is $\nu$ reflected and its $0$s
	and $1$s are switched; 
	after that, all $10$s in $\nu\overline{\nu}$ are turned into $1$s, e.g.
	$$ 0,10,1,0,10 \quad\mapsto\quad 
	0,10,1,0,10,\, 10,1,0,10,1 \quad\mapsto\quad 0,1,1,0,1,\, 1,1,0,1,1 $$
	The bijective map $coord$ takes a $0,1$-sequence $\lambda$ to the coordinate 
	$k$-plane that uses the coordinates in the $0$ positions of $\lambda$
	(so, $1,4,8$ in the above example).
	Note that $coord \circ\wt{\iota}(\nu) \in SpGr(k,2n)$ by the
	antidiagonality we required of the Gram matrix.
	
	The right-hand square, and the inclusion $T^n\into T^{2n}$, 
	induce the ring homomorphisms 
	\[
	\begin{tikzcd}
	H^*_{T^{2n}}(Gr(k,2n)) \arrow{r}{f_1} \arrow{d}[swap]{g_1} 
	& H^*_{T^n}(Gr(k,2n)) \arrow{r}{f_2\ =\ \iota^*} \arrow{d}{g_2} 
	& H^*_{T^n}(SpGr(k,2n)) \arrow{d}{g_3} \\   
	H^*_{T^{2n}}(Gr(k,2n)^{T^{2n}}) \arrow{r}{h_1} 
	& H^*_{T^n}(Gr(k,2n)^{T^n}) \arrow{r}{h_2} 
	& H^*_{T^n}(SpGr(k,2n)^{T^n})
	\end{tikzcd}
	\]
	and since each $g_i$ is injective (see e.g. \cite{Kirwan}), 
	we can compute along the bottom row, which is the proof technique used
	in \cite{KZJ} and \S \ref{sec:proof}.
	On each of our flag manifolds, we define our Schubert classes 
	as associated to the closures of orbits of $B_-$ or $B_- \cap Sp_{2n}$.

	\section{Scattering diagrams and their 
		tensor calculus}\label{sec:TC}
	\newcommand\tensor\otimes
	
	In the statement and proof of Theorem \ref{thm:generalk}, we work
	with \textit{half-puzzles}, i.e., labeled half-triangles $2n\hp$
	of size $2n$, tiled with the triangle and rhombus puzzle pieces described in \S \ref{ssec:duality}, as well as half-rhombus puzzle
	pieces obtained by cutting the existing self-dual ones vertically in half.  
	As discussed earlier, a half-puzzle can be considered as half of a
	self-dual puzzle with all three sides of length $2n$. 
	In our notation, a ``rhombus'' can
	also be made of a $\Delta$ and a $\nabla$ triangle glued together.
	
	To linearize the puzzle pictures and relate them back to the
	restriction of cohomology classes, we consider the puzzle labels
	$\{0,10,1\}$ as indexing bases for three spaces
	$\C^3_G, \C^3_R, \C^3_B$ (Green, Red, Blue). In our scattering diagrams below, each coloured edge will carry its
	corresponding vector space.

	\begin{enumerate}[leftmargin=0.2in]
		\item Take an \textbf{unlabeled} size $2n$ half-puzzle triangle $2n\hp$ tiled by rhombi, half-rhombi (on the East) and triangles (on the South) as before,
		with assigned ``spectral parameters'' $y_1,\hdots,y_n, -y_n,\hdots,-y_1$
		on the Northwest side.
		\vspace{0.3cm}
		
		\item 		\parpic[r][b]{%
			\begin{minipage}[b]{50mm}\vspace{-0.5cm}
				\[\begin{tikzpicture}[baseline=1.5cm]
				\foreach \i in {1,...,4} \draw[dg,arrow=1.2*\i/(\i+1.5)] (\i+0.5,0.5) -- node[left=-1mm,pos=1.1*\i/(\i+1),black] {$\sss y_{\i}$}++(-\i/2,\i/2);
				\foreach \i in {1,...,4} \draw[dg,arrow=1.2*\i/(\i+1.5)] (5,5-\i) -- node[left=-1mm,pos=1.1*\i/(\i+1),black] {$\sss -y_{\i}$}++(-\i/2,\i/2);
				\foreach \i in {1,...,4} \draw[dr,arrow=0.6*\i/(\i+0.2)] (\i+0.5,0.5) -- (5,5-\i);
				\draw[db,invarrow=0.3,invarrow=0.8] (3.5,0.5) -- node[left,pos=0.85,black] {$\sss y_3$}(3.5,-0.5) (2.5,0.5) -- node[left,pos=0.85,black] {$\sss y_2$}(2.5,-0.5);
				\draw[db,invarrow=0.3,invarrow=0.8] (1.5,0.5) -- node[left,pos=0.85,black] {$\sss y_1$}(1.5,-0.5) (4.5,0.5) -- node[left,pos=0.85,black] {$\sss y_4$}(4.5,-0.5);
				\draw[dashed,d] (5,-0.5) -- (5,4.75);
				\end{tikzpicture}\]
			\end{minipage}
		}
		
		Consider the dual-graph picture of strands, oriented upwards.
		Each rhombus corresponds to a crossing of two strands, each
		half-rhombus 
		to a bounce off the East wall and negates the spectral parameter,
		and each triangle 
		to a trivalent vertex with all parameters equal.

		We also colour the Northwest-pointing strands green, Northeast-pointing red,
		and North-pointing blue.
		
		\vspace{0.5cm}
		
		\item We let $a$ and $b$ denote two spectral parameters from Step 1. We assign the following linear maps
		\vspace{0.2cm}
		\begin{itemize}[leftmargin=0.2in]
			\item to each crossing of two strands with left and right parameters
			$a$ and $b$, and colours $C$ and $D$, a linear map
			$R_{CD}(a-b): \C_C^3 \tensor \C_D^3 \longrightarrow \C_D^3 \tensor \C_C^3$;
			\item to each wall-bounce of a colour $C$ strand with parameter $a$ bouncing to $-a$, a
			map $K_C(a):\C^3 \rightarrow \C^3$;
			\item to each trivalent vertex with incoming blue strand and outgoing green and red strands, all with
			parameters $a$, a map
			$U(a): \C^3_B \longrightarrow \C^3_G \tensor \C^3_R$.
		\end{itemize}
	\end{enumerate}
	
	Connecting two strands corresponds to composing the corresponding maps, so the whole $2n\hp$ corresponds to a linear map
	$\Phi: (\C^3_B)^{\tensor n} \longrightarrow (\C^3_G)^{\tensor 2n}$.
	
	\newcommand\uptr[4][]{ \tikzif[baseline=0.34cm,scale=1.25]{ \getshift
			\draw[thick,black] (-0.5,0) -- node[edgelabel] (horiz)
			{$\vphantom{\noparen{#3}}\smash{#3}$} ++(0:1) --
			node[edgelabel,right,xshift=-\myshift] (NE) {$#4$} ++(120:1) --
			node[edgelabel,left,xshift=\myshift] (NW) {$#2$} ++(240:1) --
			cycle; 
			\ifx\&#1\&\else\node[invlabel] at (0,0.33) {$\sss#1$};\fi }}
	
	\newcommand\downtr[4][]{ \tikzif[baseline=-0.54cm,scale=1.25]{
			\getshift\begin{scope}[scale=-1]\draw[thick,black] (-0.5,0) --
				node[edgelabel] (horiz) {$\vphantom{\noparen{#3}}\smash{#3}$} ++(0:1) -- 
				node[edgelabel,left,xshift=\myshift] (NE) {$#4$} ++(120:1) --
				node[edgelabel,right,xshift=-\myshift] (NW) {$#2$} ++(240:1) --
				cycle; 
				\ifx\&#1\&\else\node[invlabel] at (0,0.33) {$\sss#1$};\fi
			\end{scope}
	}}
	
	\begin{defn}[The $R$-, $U$-, and $K$-matrices]\label{def:RKU} In terms
		of the bases of $\C^3_G,\C^3_R,\C^3_B$ indexed by $\{0,10,1\}$, the
		above sparse matrices can be written compactly as follows 
		(where a labeled diagram corresponds to the coefficient of the map in those
		basis elements):
		\begin{align*}
		R_{CC}(a-b): 	
		&\tikz[baseline=0,xscale=0.5]
		{\draw[arrow=0.25,d] (-0.75,-0.5) node[below] {$\m k$} -- (0.75,0.5) node[above] {$\m j$};
			\draw[arrow=0.25,d] (0.75,-0.5) node[below] {$\m l$} -- (-0.75,0.5) node[above] {$\m i$};}=
		\begin{cases}1 & \text{ if } (i,j)=(k,l), \\
		b-a & \text{ if } (i,j,k,l) \in \{(1,0,0,1),(10,0,0,10),(1,10,10,1)\} \end{cases}
		\\
		& \text{ where } C \in \{R,G,B\} \text{ and the two strands are any identical colour} 
		\\
		R_{RG}(a-b): 
		&\tikz[baseline=0,xscale=0.5]
		{\draw[arrow=0.25,dg] (0.75,-0.5) node[below] {$\m l$} -- (-0.75,0.5) node[above] {$\m i$};
			\draw[arrow=0.25,dr] (-0.75,-0.5) node[below] {$\m k$} -- (0.75,0.5) node[above] {$\m j$};}=
		\begin{cases}a-b & \text{ if } (i,j,k,l)=(0,1,1,0), \\
		1 & \text{ if } (i,j,k,l) \in 
		\left\{ \begin{matrix}
		0^4,1^4,0^2 1 (10),0 (10) 1^2,0 (10)^2 0, \\
		1\ 0^2 1,1^2 (10) 0,(10) 1\ 0^2,(10) 1^2(10)
		\end{matrix} \right\}
		\end{cases} 
		\end{align*}
		\begin{align*}
		K_R(a):& \tikz[baseline=0,xscale=0.5]
		{\draw[arrow=0.6,dg] (0,0) -- (-0.75,0.45) node[above] {$\m i$};
			\draw[arrow=0.6,dr] (-0.75,-0.45) node[below] {$\m j$} -- (0,0);
			\draw[dashed,thick] (0,0.5) -- (0,-0.5) ;} =	1 \text{ if } (i,j) \in \{(1,0),(0,1)\} \quad \quad 
		K_B(a): \tikz[baseline=0,xscale=0.5]
		{\draw[arrow=0.6,db] (0,0) -- (-0.75,0.45) node[above] {$\m i$};
			\draw[arrow=0.6,db] (-0.75,-0.45) node[below] {$\m j$} -- (0,0);
			\draw[dashed,thick] (0,0.5) -- (0,-0.5) ;}=	\begin{cases}1 & \text{ if } i=j, \\
		-2a & \text{ if } (i,j)=(1,0) \end{cases} \\
		U(a): & \tikz[baseline=0,xscale=0.5]
		{\draw[arrow=0.6,dg] (0,0) -- (-0.75,0.5) node[above] {$\m i$};
			\draw[arrow=0.6,dr] (0,0) -- (0.75,0.5) node[above] {$\m j$};
			\draw[invarrow=0.5,db] (0,0) -- (0,-0.5) node[right] {$\m k$};
		}=1 \text{ if }\; (i,j,k) \in \{(0,0,0),(0,10,1),(1,0,10),(1,1,1),(10,1,0)\}
		\end{align*}
		
		The subscripts $R,G,B$ on the maps indicate the colours of
		the incoming edges (listed counterclockwise). For each map, the matrix entries which are not listed are zero.
	\end{defn}
	
	Note that if we take the corresponding bases with lexicographic
	ordering, with alphabet ordered as $\{0,10,1\}$, then the matrices for
	$R_{CC}$ and $K_B$ are lower-triangular. See \cite{ZJ} for these $R$-matrices and \cite[\S 3]{KZJ} for their representation-theoretic origins.
	
	\begin{defn}\label{def:fug}
		With the above notation, let $\mathbf{P}$ be a half-puzzle with
		boundary labels $\halfuptri{\la}{\nu}$, where $\la \in 0^k1^{2n-k}$
		and $\nu \in (10)^{n-k}\{0,1\}^k$. 
		The fugacity $\fug(\mathbf{P})$ of $\mathbf{P}$ is the product over all puzzle pieces (dually: vertices) of the entries of
		the corresponding $R$-, $U$-, $K$-matrices.
	\end{defn}
	In this way, the summation over half-puzzles reproduces the full matrix product, i.e.,
	$
	\text{ the $(\la,\nu)$ matrix entry of $\Phi$} = 
	\sum_{\mathbf{P}} \Big\{\fug(\mathbf{P}) \; | \; \mathbf{P} \text{ is a puzzle with boundary \halfuptri{\la}{\nu} }\Big\}
	$.
	
	\begin{lem}\label{lem:TCP} The matrices defined in \S\ref{sec:TC} satisfy the following identities:
		\begin{enumerate}[i),leftmargin=0.2in]
			\item The Yang--Baxter equation.
			\begin{align}
			\begin{tikzpicture}[baseline=-3pt,y=2cm,scale=0.8]
			\draw[invarrow=0.2,rounded corners,dg] (-0.5,0.5) node[left=-1mm,black] {$\sss u_1$} -- (0.75,0) -- (1.5,-0.5) node[left=-1mm,black] {$\sss u_1$};
			\draw[invarrow=0.2,rounded corners,dr] (0.5,0.5) node[left=-1mm,black] {$\sss u_2$} -- (0.25,0) -- (0.5,-0.5) node[left=-1mm,black] {$\sss u_2$};
			\draw[invarrow=0.2,rounded corners,dr] (1.5,0.5) node[left=-1mm,black] {$\sss u_3$} -- (0.75,0) -- (-0.5,-0.5) node[left=-1mm,black] {$\sss u_3$};
			\end{tikzpicture}
			&=\begin{tikzpicture}[baseline=-3pt,y=2cm,scale=0.8]
			\draw[invarrow=0.2,rounded corners,dg] (-0.5,0.5) node[left=-1mm,black] {$\sss u_1$} -- (0.25,0) -- (1.5,-0.5) node[left=-1mm,black] {$\sss u_1$};
			\draw[invarrow=0.2,rounded corners,dr] (0.5,0.5) node[left=-1mm,black] {$\sss u_2$} -- (0.75,0) -- (0.5,-0.5) node[left=-1mm,black] {$\sss u_2$};
			\draw[invarrow=0.2,rounded corners,dr] (1.5,0.5) node[left=-1mm,black] {$\sss u_3$} -- (0.25,0) -- (-0.5,-0.5) node[left=-1mm,black] {$\sss u_3$};
			\end{tikzpicture}
			\hspace{2cm}
			\begin{tikzpicture}[baseline=-3pt,y=2cm,scale=0.8]
			\draw[invarrow=0.2,rounded corners,dg] (-0.5,0.5) node[left=-1mm,black] {$\sss u_1$} -- (0.75,0) -- (1.5,-0.5) node[left=-1mm,black] {$\sss u_1$};
			\draw[invarrow=0.2,rounded corners,dg] (0.5,0.5) node[left=-1mm,black] {$\sss u_2$} -- (0.25,0) -- (0.5,-0.5) node[left=-1mm,black] {$\sss u_2$};
			\draw[invarrow=0.2,rounded corners,dr] (1.5,0.5) node[left=-1mm,black] {$\sss u_3$} -- (0.75,0) -- (-0.5,-0.5) node[left=-1mm,black] {$\sss u_3$};
			\end{tikzpicture}
			&=\begin{tikzpicture}[baseline=-3pt,y=2cm,scale=0.8]
			\draw[invarrow=0.2,rounded corners,dg] (-0.5,0.5) node[left=-1mm,black] {$\sss u_1$} -- (0.25,0) -- (1.5,-0.5) node[left=-1mm,black] {$\sss u_1$};
			\draw[invarrow=0.2,rounded corners,dg] (0.5,0.5) node[left=-1mm,black] {$\sss u_2$} -- (0.75,0) -- (0.5,-0.5) node[left=-1mm,black] {$\sss u_2$};
			\draw[invarrow=0.2,rounded corners,dr] (1.5,0.5) node[left=-1mm,black] {$\sss u_3$} -- (0.25,0) -- (-0.5,-0.5) node[left=-1mm,black] {$\sss u_3$};
			\end{tikzpicture}\label{eqn:YBRG} \\
			\begin{tikzpicture}[baseline=-3pt,y=2cm,scale=0.8]
			\draw[invarrow=0.075, invarrow=0.42, invarrow=0.7,rounded corners,dg] (-0.5,0.5) node[left=-1mm,black] {$\sss u_1$} -- (0.75,0) -- (1.5,-0.5) node[left=-1mm,black] {$\sss u_1$} (0.5,0.5) node[left=-1mm,black] {$\sss u_2$} -- (0.25,0) -- (0.5,-0.5) node[left=-1mm,black] {$\sss u_2$} (1.5,0.5) node[left=-1mm,black] {$\sss u_3$} -- (0.75,0) -- (-0.5,-0.5) node[left=-1mm,black] {$\sss u_3$};
			\end{tikzpicture}
			&=\begin{tikzpicture}[baseline=-3pt,y=2cm,scale=0.8]
			\draw[invarrow=0.075, invarrow=0.42, invarrow=0.7,rounded corners,dg] (-0.5,0.5) node[left=-1mm,black] {$\sss u_1$} -- (0.25,0) -- (1.5,-0.5) node[left=-1mm,black] {$\sss u_1$} (0.5,0.5) node[left=-1mm,black] {$\sss u_2$} -- (0.75,0) -- (0.5,-0.5) node[left=-1mm,black] {$\sss u_2$} (1.5,0.5) node[left=-1mm,black] {$\sss u_3$} -- (0.25,0) -- (-0.5,-0.5) node[left=-1mm,black] {$\sss u_3$};
			\end{tikzpicture}
			\hspace{2cm}
			\begin{tikzpicture}[baseline=-3pt,y=2cm,scale=0.8]
			\draw[invarrow=0.075, invarrow=0.42, invarrow=0.7,rounded corners,db] (-0.5,0.5) node[left=-1mm,black] {$\sss u_1$} -- (0.75,0) -- (1.5,-0.5) node[left=-1mm,black] {$\sss u_1$} (0.5,0.5) node[left=-1mm,black] {$\sss u_2$} -- (0.25,0) -- (0.5,-0.5) node[left=-1mm,black] {$\sss u_2$} (1.5,0.5) node[left=-1mm,black] {$\sss u_3$} -- (0.75,0) -- (-0.5,-0.5) node[left=-1mm,black] {$\sss u_3$};
			\end{tikzpicture}
			&=\begin{tikzpicture}[baseline=-3pt,y=2cm,scale=0.8]
			\draw[invarrow=0.075, invarrow=0.42, invarrow=0.7,rounded corners,db] (-0.5,0.5) node[left=-1mm,black] {$\sss u_1$} -- (0.25,0) -- (1.5,-0.5) node[left=-1mm,black] {$\sss u_1$} (0.5,0.5) node[left=-1mm,black] {$\sss u_2$} -- (0.75,0) -- (0.5,-0.5) node[left=-1mm,black] {$\sss u_2$} (1.5,0.5) node[left=-1mm,black] {$\sss u_3$} -- (0.25,0) -- (-0.5,-0.5) node[left=-1mm,black] {$\sss u_3$};
			\end{tikzpicture}\label{eqn:YBGGBB}
			\end{align}
			For example, the linear map form of the Northwest equation is
			\begin{align*}
			&     
			(R_{RG}(u_2-u_1) \tensor \Id) 
			\circ
			(\Id \tensor R_{RG}(u_3-u_1)) 
			\circ
			(R_{RR}(u_3-u_2) \tensor \Id) 
			\\ 
			=\,&
			(\Id \tensor R_{RR}(u_3-u_2)) 
			\circ
			(R_{RG}(u_3-u_1) \tensor \Id) 
			\circ
			(\Id \tensor R_{RG}(u_2-u_1)) 
			\end{align*}  
			\item Swapping of two trivalent vertices.
			\begin{equation}\label{eqn:triv}
			\begin{tikzpicture}[baseline=0]
			\foreach \i in {1,2} \draw[dg,arrow=0.6] (\i,0) -- node[left=-1mm,pos=0.8,black] {$\sss u_{\i}$} ++(-1.25,1.25);
			\foreach \i in {1,2} \draw[dr,arrow=0.6] (\i,0) -- node[right=-1mm,pos=0.8,black] {$\sss u_\i$} ++(1.25,1.25);
			\draw[rounded corners,db,invarrow=0.3] (1,0) -- (1,-0.25) -- (2,-0.75) -- node[right,pos=0.5,black] {$\sss u_1$} (2,-1);
			\draw[rounded corners,db,invarrow=0.3] (2,0) -- (2,-0.25) -- (1,-0.75) -- node[left,pos=0.5,black] {$\sss u_2$}(1,-1);
			\end{tikzpicture}=
			\begin{tikzpicture}[baseline=0]
			\draw[dg,arrow=0.4,rounded corners] (1,0) -- ++(-0.25,0.25) -- node[pos=0.9,left=-2mm,black] {$\sss u_2$} ++(0,1);
			\draw[dg,arrow=0.4,rounded corners] (2,0) -- ++(-0.5,0.5) -- ++(-1.5,0.5) -- node[pos=0.4,left=-2mm,black] {$\sss u_1$} ++(-0.05,0.05);
			\draw[dr,arrow=0.4,rounded corners] (1,0) -- ++(0.5,0.5) -- ++(1.5,0.5) -- node[pos=0.2,right=-1mm,black] {$\sss u_2$} ++(0.05,0.05);
			\draw[dr,arrow=0.4,rounded corners] (2,0) -- ++(0.25,0.25) -- node[pos=0.9,right=-1mm,black] {$\sss u_1$} ++(0,1);
			\draw[db,invarrow=0.3,invarrow=0.8] (1,0) -- node[pos=0.85,right] {$\sss u_2$} (1,-1) (2,0) -- node[pos=0.85,right] {$\sss u_1$} (2,-1);
			\end{tikzpicture}
			\end{equation}
			\begin{align*}
			&  
			(\Id \tensor R_{RG}(u_1-u_2) \tensor \Id) 
			\circ 
			(U(u_1) \tensor U(u_2)) 
			\circ 
			R_{BB}(u_2-u_1) 
			\\
			=\,& 
			(R_{GG}(u_2-u_1) \otimes R_{RR}(u_2-u_1))
			\circ 
			(\Id \otimes R_{RG}(u_2 - u_1) \otimes \Id) 
			\circ 
			(U(u_2)\tensor U(u_1)) 
			\end{align*}
			\item The reflection equation.
			\begin{align}\label{eqn:reflection}
			\begin{tikzpicture}[baseline=0.4cm,scale=0.8]
			\draw[dg,arrow=0.6] (0,0) -- node[left=-1mm,pos=0.9,black] {$\sss u_1$} ++(-2,2);
			\draw[dr,invarrow=0.4] (0,0) -- node[left=-1mm,pos=0.8,black] {$\sss -u_1$} ++(-1.5,-1);
			\draw[dg,arrow=0.6] (0,1) -- node[left=-1mm,pos=0.8,black] {$\sss u_2$} ++(-1,1);
			\draw[dr,rounded corners,invarrow=0.4] (0,1) -- (-1.5,0) -- node[left=-1mm,pos=0.2,black] {$\sss -u_2$} ++(0.25,-0.25) -- (-0.5,-1);
			\draw[dashed,d] (0,2) -- (0,-1);
			\end{tikzpicture}
			\hspace{0.25cm}=\hspace{-0.25cm}
			\begin{tikzpicture}[baseline=0.4cm,scale=0.8]
			\draw[dr,invarrow=0.6] (0,0) -- node[right=-1mm,pos=0.8,black] {$\sss -u_2$} ++(-1,-1);
			\draw[dg,rounded corners,arrow=0.4] (0,0) -- (-1.5,1) -- node[right=-1mm,pos=0.9,black] {$\sss u_2$} ++(0.75,1) ;
			\draw[dr,invarrow=0.6] (0,1) -- node[right=-1mm,pos=0.9,black] {$\sss -u_1$} ++(-2,-2);
			\draw[dg,arrow=0.4] (0,1) -- node[left=-1mm,pos=0.9,black] {$\sss u_1$} ++(-1.5,1);
			\draw[dashed,d] (0,2) -- (0,-1);
			\end{tikzpicture}
			\hspace{1.25cm}
			\begin{tikzpicture}[baseline=0.4cm,scale=0.8]
			\draw[db,arrow=0.6] (0,0) -- node[left=-1mm,pos=0.9,black] {$\sss u_1$} ++(-2,2);
			\draw[db,invarrow=0.4] (0,0) -- node[left=-1mm,pos=0.8,black] {$\sss -u_1$} ++(-1.5,-1);
			\draw[db,arrow=0.6] (0,1) -- node[left=-1mm,pos=0.8,black] {$\sss u_2$} ++(-1,1);
			\draw[db,rounded corners,invarrow=0.4] (0,1) -- (-1.5,0) -- node[left=-1mm,pos=0.2,black] {$\sss -u_2$} ++(0.25,-0.25) -- (-0.5,-1);
			\draw[dashed,d] (0,2) -- (0,-1);
			\end{tikzpicture}
			\hspace{0.25cm}=\hspace{-0.25cm}
			\begin{tikzpicture}[baseline=0.4cm,scale=0.8]
			\draw[db,invarrow=0.6] (0,0) -- node[right=-1mm,pos=0.8,black] {$\sss -u_2$} ++(-1,-1);
			\draw[db,rounded corners,arrow=0.4] (0,0) -- (-1.5,1) -- node[right=-1mm,pos=0.9,black] {$\sss u_2$} ++(0.75,1) ;
			\draw[db,invarrow=0.6] (0,1) -- node[right=-1mm,pos=0.9,black] {$\sss -u_1$} ++(-2,-2);
			\draw[db,arrow=0.4] (0,1) -- node[left=-1mm,pos=0.9,black] {$\sss u_1$} ++(-1.5,1);
			\draw[dashed,d] (0,2) -- (0,-1);
			\end{tikzpicture}		
			\end{align}
			In linear map terms, the left equation says
			\begin{align*}
			&
			(\Id \tensor K_R(-u_2)) 
			\circ 
			R_{RG}(-u_2-u_1) 
			\circ  
			(\Id \tensor K_R(-u_1)) 
			\circ 
			R_{RR}(-u_1+u_2) 
			\\
			=\,&
			R_{GG}(u_2-u_1)
			\circ 
			(\Id \tensor K_R(-u_1)) 
			\circ 
			R_{RG}(-u_1-u_2) 
			\circ 
			(\Id \tensor K_R(-u_2)) 
			\end{align*}
			
			\item $K$-fusion.
			\begin{equation}\label{eqn:fusion}
			\begin{tikzpicture}[baseline=.7cm]
			\draw[dg,arrow=0.6] (-0.5,0.5) -- node[left=-1mm,pos=0.9,black] {$\sss u_1$} ++(-1.5,1.5);
			\draw[dg,arrow=0.6] (0,1) -- node[left=-1mm,pos=0.9,black] {$\sss -u_1$} ++(-1,1);
			\draw[dr,invarrow=0.6] (0,1) -- node[left=0mm,pos=0.2,black] {$\sss u_1$} ++(-0.5,-0.5);
			\draw[db,rounded corners,invarrow=0.5] (-0.5,0.5) -- (-0.5,0.2)-- node[left=-1mm,pos=0.1,black] {$\sss u_1$} ++ (0.5,-0.4);
			\draw[db,invarrow=0.5] (0,-0.2)-- node[left=0mm,pos=0.55,black] {$\sss -u_1$} (-0.5,-0.7);
			\draw[dashed,d] (0,2) -- (0,-0.7);
			\end{tikzpicture}
			\hspace{0.25cm}=\hspace{-0.25cm}
			\begin{tikzpicture}[baseline=.7cm]
			\draw[dg,rounded corners,arrow=0.4] (-0.65,0.35) -- (-1.5,1) -- node[right=-1mm,pos=0.9,black] {$\sss -u_1$} ++(0.75,1) ;
			\draw[dr,invarrow=0.6] (0,1) -- node[left=-1mm,pos=0.2,black] {$\sss -u_1$} ++(-0.65,-0.65);
			\draw[dg,arrow=0.4] (0,1) -- node[left=-1mm,pos=0.9,black] {$\sss u_1$} ++(-1.5,1);
			\draw[db,invarrow=0.5] (-0.65,0.35) -- node[left=0mm,pos=0.8,black] {$\sss -u_1$} ++ (0,-1.05);
			\draw[dashed,d] (0,2) -- (0,-0.7);
			\end{tikzpicture}
			\qquad\qquad
			\begin{array}{c}
			(\Id \tensor K_R(u_1)) 
			\circ 
			U(u_1) 
			\circ 
			K_B(-u_1) 
			\\ = \\
			R_{GG}(-2u_1)
			\circ 
			(\Id \tensor K_R(-u_1)) 
			\circ 
			U(-u_1) 
			\end{array}
			\end{equation}
		\end{enumerate}
	\end{lem}

	\section{AJS/Billey formul\ae\ as scattering diagrams}
	\label{sec:AJSBilley}
	
	We first discuss the general AJS/Billey formula for restricting an equivariant Schubert class to a torus-fixed point, and then consider the special cases of types $A$ and $C$. 
	Let $G$ be an algebraic group and fix a pinning $G \geq B \geq T$,
	with $W_G = N_G(T)/T$. Let $B_-$ denote the opposite Borel and $P\geq B$
	a parabolic, with Weyl group $W_P$. We recall that Schubert classes are indexed by $W_G/W_P$,
	which we identify with strings (or signed strings in type $C$), on which $W_G$ acts by permuting/negating positions: $\pi W_P \mapsto \omega \circ \pi^{-1}$
	(see prop. \ref{prop:ajsbwiring} for $\omega$).
	In particular for $P=B$ our indexing is inverse to the usual convention;
	this inversion is forced on us by the necessary use 
	for general $P$ of strings-with-repeats,
	e.g. binary strings rather than Grassmannian permutations.
	
	\begin{prop}\label{prop:AJSB}\begin{enumerate}
			\item (\cite{AJS,Billey}) For the Schubert class
			$S_{\pi} := \left[\overline{B_- \pi B}/B\right] \in H^{\ast}_T(G/B)$, \\
			$\pi,\sigma \in W_G$, and $Q=(q_1,\hdots,q_k)$ a reduced word in
			simple reflections with $\prod Q=\sigma$, 
			the AJS/Billey formula tells us that
			\begin{equation}\label{eqn:AJSB}
			S_{\pi}|_{\sigma}=\sum_{\small\begin{array}{c} R \subseteq Q\\ \prod R=\pi\end{array}} 
			\prod^k_{i=1}(\widehat{\alpha_{q_i}}^{\left[q_i \in
				R\right]}q_i)\cdot 1=\sum_{\small\begin{array}{c} R \subseteq Q\\ 
				\prod R=\pi\end{array}}\prod_{i \in R}\beta_i \; \in \;
			H^{\ast}_T(\text{pt}) \end{equation} where $\beta_i := q_1q_2 \hdots
			q_{i-1}\cdot\alpha_{q_i}$ and the summation is over reduced subwords
			$R$ of $Q$.
			\item 
			To compute a point restriction $S_\lambda|_\mu$ on $G/P$, 
			where $\lambda,\mu \in W_G/W_P$, we use
			lifts $\wt{\la},\wt{\mu} \in W_G$ such that $\wt{\la}$ is the
			shortest length representative of $\la$, and observe that
			$ S_{\la}|_{\mu} = S_{\raisebox{0pt}[0pt]{$\sss\wt{\la}$}}|_{\wt{\mu}} $.
		\end{enumerate}
	\end{prop}
	
	\junk{
		\begin{rem}
			\begin{enumerate}[i)]
				\item Note that Equation (\ref{eqn:AJSB})
				still holds if $Q$ is not reduced, but choosing it in this way
				guarantees that the $\beta_i$ are positive roots which occur with
				multiplicity at most one, i.e. this formula is positive in the sense
				of \cite{Graham}.
				\item Parts (1) and (2) of Proposition
				\ref{prop:AJSB} together give us the AJS/Billey formula for partial
				flag varieties $G/P$. \icom{cite Kostant-Kumar}
			\end{enumerate}
	\end{rem} }
	
	Below we give a diagrammatic description of the formula from
	Proposition \ref{prop:AJSB} in the cases when $(G,W_G)$ is $(GL_{2n}, S_{2n})$ or $(Sp_{2n},S_n \ltimes (\Z/2\Z)^n)$ using the tensor calculus setting of \S \ref{sec:TC}. We first introduce some notation.

	Consider $\sigma$ in $W_G$ (generated by simple reflections $\{s^G_i\}$) and a reduced word for $\sigma$, $Q_{\sigma}=(q_1,\hdots,q_k)$ (where $q_i=s_{p_i}$ for some
	$p_i$). We can associate to it a \textit{wiring diagram}
	$D(Q_{\sigma})$ by assigning the diagrams below to the simple reflections $\{s^G_i~:~1\leq i \leq m_G-1\}$ for $(m_{GL_{2n}},m_{Sp_{2n}})=(2n,n)$, and to $s^{Sp_{2n}}_n$ respectively. Then, a word in simple reflections corresponds to a concatenation of such diagrams.   \vspace{-0.2cm}
	\[s^G_i \mapsto \tikz[baseline=0,xscale=0.5] {\draw[invarrow=0.5,d]
		(-4,0.5) node[above] {$\m 1$} -- (-4,-0.5); \draw (-3.75,0) node[right] {$\hdots$};
		\draw[invarrow=0.5,d] (-2,0.5) -- (-2,-0.5); \draw[invarrow=0.5,d] (2,0.5) 
		-- (2,-0.5); \draw (2.25,0) node[right] {$\hdots$}; \draw[invarrow=0.5,d]
		(4,0.5) node[above] {$\m m_G$} -- (4,-0.5); \draw[invarrow=0.3,d] (-1,0.5) node[above] {$\m i$}
		-- (1,-0.5); \draw[invarrow=0.3,d] (1,0.5) node[above] {$\m i+1$} --
		(-1,-0.5);}, \; \text{ for } 1 \leq i \leq m_G-1 \quad \quad \text{ and
	} \quad \quad s^{Sp_{2n}}_n \mapsto \tikz[baseline=0,xscale=0.5]
	{\draw[invarrow=0.5,d] (-4,0.5) node[above] {$\m 1$} -- (-4,-0.5); \draw (-3.75,0) node[right]
		{$\hdots$}; \draw[invarrow=0.5,d] (-2,0.5) -- (-2,-0.5);
		\draw[invarrow=0.3,d] (-1,0.5) node[above] {$\m n$} -- (0,0) --
		(-1,-0.5); \draw[dashed,d] (0,0.5) -- (0,-0.5);}
	\] Each wire in a wiring diagram is also assigned a spectral
	parameter. For $G=GL_{2n}$, they are $y_1,\hdots, y_{2n}$ along the top (which we need to later specialize to $y_1,\hdots,y_n, -y_n,\hdots,-y_1$ as in the maps $f_1,h_1$ in \S \ref{sec:setup}), and for $G=Sp_{2n}$ they are $y_1,\hdots,y_n$. 
	
	In the context of the tensor calculus from \S \ref{sec:TC}, the wiring
	diagram $D(Q_{\sigma})$ can be interpreted as a scattering diagram,
	i.e., giving a map $(\C^3_C)^{\otimes m_G} \rightarrow (\C^3_C)^{\otimes m_G}$;
	we replace each crossing with $R_{GG}$ in the $GL_{2n}$ (and $C=G$) case
	or with $R_{BB}$ in the $Sp_{2n}$ (and $C=B$) case, 
	and replace each bounce with $K_B$ which also negates the spectral parameter.
	For instance, take $G=Sp_6$ and $\sigma=31\bar2$, $Q_{\sigma}=(s_2,s_3,s_1)$, then
	\vspace{-0.2cm}
	\[D(Q_{\sigma})=\begin{tikzpicture}[baseline=0cm]
	\draw[db,invarrow=0.4] (-2,1) node[left=-1mm,black] {$\sss y_1$}  -- (-1.5,-1) node[left=-1.5mm,black] {$\sss y_1$};
	\draw[db,invarrow=0.4] (-1.5,1) node[left=-1.5mm,black] {$\sss y_2$} -- (-0.25,0.25) -- (-0.75,-1) node[left=-1.5mm,black] {$\sss -y_2$};
	\draw[db,invarrow=0.4] (-0.75,1) node[left=-1mm,black] {$\sss y_3$} -- (-2,-1) node[left=-1mm,black] {$\sss y_3$};
	\draw[dashed,d] (-0.25,1) -- (-0.25,-1);
	\end{tikzpicture} \quad \quad 
	\begin{array}{c}
	(\Id \otimes R_{BB}(y_3 - y_2))
	\circ 
	(Id^{\otimes 2} \otimes K_B(-y_2)) 
	\circ 
	(R_{BB}(y_3 - y_1) \otimes \Id) 
	: \\ (\C^3_B)^{\otimes 3} \rightarrow (\C^3_B)^{\otimes 3}\end{array}\]
	
	\begin{prop}\label{prop:ajsbwiring}
		Let $\lambda,\mu$ be strings in $0,10,1$ as in \S\ref{sec:setup},
		which we identify with cosets $W_G/W_P$
		where $W_G$ is of type $C$ and $P$ is maximal, or
		of type $A$ and $P$ is maximal or submaximal. 
		Let $\omega_{Gr}=0\hdots 0 \ 1\hdots 1 \in 0^k 1^{2n-k}$ for
		$G/P = Gr(k,\C^{2n})$,
		$\omega_{SpGr} = 0\hdots 0 \ 10 \hdots 10 \in 0^k (10)^{n-k}$ for
		$G/P = SpGr(k,\C^{2n})$, or
		$\omega_{Fl} = 0\hdots 0 \ 10 \hdots 10 \ 1 \hdots 1 \in 0^j
		(10)^{k-j} 1^{2n-k}$ for $G/P = Fl(j,k;\, \C^{2n})$.
		Make a wiring diagram as just explained, 
		using a reduced word for the shortest lift $\wt{\mu}$; interpret it as
		a scattering diagram map, using the $R_{BB}(=R_{GG})$ matrix
		for crossings and (in type $C$) $K_B$ for bounces. 
		Then $S_{\lambda}|_{\mu}$ is the $(\lambda,\omega_{G/P})$ 
		matrix entry of the resulting product.
	\end{prop}
	
	The essentially routine rewriting of Proposition \ref{prop:AJSB}
	to give Proposition \ref{prop:ajsbwiring} will appear elsewhere.
	The principal thing one checks is that $R_{BB}$ is the correct $R$-matrix
	for three labels $\{0,10,1\}$. In view of Proposition \ref{prop:ajsbwiring}, for $\la,\mu,\nu \in W_G/W_P$ as above, we denote
	\[\begin{tikzpicture}[baseline=-5.5mm, math mode,scale=1.2]
	\draw[thick] (0,-0.6) -- node[pos=0.5] {\nu} (2,-0.6); 
	\draw (0,0) -- (0,-0.6); \draw[dashed] (2,0) -- (2,-0.6); 
	\draw (0,0) -- node[pos=0.5] {\la} (2,0); \node at (1,-0.3) {\mu};
	\end{tikzpicture}:= \begin{array}{c}\text{the $(\la,\nu)$ matrix entry for the scattering diagram map} \\ \text{coming from a reduced word for $\widetilde{\mu}$.}\end{array}\]
	By the proposition, when $\nu=\omega_{G/P}$ this gives $S_{\la}|_{\mu}$.
	
	\section{Proof of Theorem \ref{thm:generalk}}
	\label{sec:proof}
	
	The proof of Theorem \ref{thm:jk} is very much as in \cite[\S 3]{KZJ} and
	will appear elsewhere. Theorem \ref{thm:n2n} is the $k=n$ special case
	of Theorem \ref{thm:generalk}. In fact, we give a more general puzzle rule for equivariant cohomology in Theorem \ref{thm:main}, which in particular implies Theorem \ref{thm:generalk}.

	\begin{customthm}{1C}\label{thm:main}
		For every $S_{\la} \in H^{\ast}_{T^n}(Gr(k,2n))$, where $\la \in 0^k1^{2n-k}$, and $\iota^{\ast}$ as in \S \ref{sec:setup}
		\[\iota^{\ast}(S_{\la})=\sum_{\nu \in (10)^{n-k}\{0,1\}^k} \left(\sum_{\mathbf{P}} \Big\{\fug(\mathbf{P}) \; | \; \mathbf{P} \text{ is a puzzle with boundary \halfuptri{\la}{\nu} }\Big\}\right)S_{\nu}\]
	\end{customthm}                                                                                                                                                                                  
	
	As explained in \S\ref{sec:setup}, it suffices to check Theorem \ref{thm:main}'s equality at each $T^n$-fixed point  $\sigma \in (10)^{n-k}\{0,1\}^k$ of $SpGr(k,2n)$. To do so, we first prove several preliminary results in the language of \S \ref{sec:TC}.

	\begin{lem}\label{lem:Kdelta} For $\omega=\omega_{SpGr}$ as in Proposition \ref{prop:ajsbwiring} and $\la \in 0^k1^{2n-k}$, we have $\halfuptri{\la}{\omega}=\delta_{\la,\omega\overline{\omega}}$.
	\end{lem}
	
	\begin{proof}
		This is a straightforward consequence of Definition \ref{def:RKU},
		when considering the $(\lambda,\omega)$ matrix entry of the product
		of $R$-, $K$-, and $U$-matrices making up the
		half-puzzle. Alternatively, note that this is half of a classical
		triangular self-dual puzzle with NW, NE, S boundaries labelled by
		$\la, \overline{\la}, \omega\overline{\omega}$, and so the result
		follows from \cite[Proposition 4]{KZJ}.
	\end{proof}

	\begin{prop}\label{prop:key} Given $\sigma \in (10)^{n-k}\{0,1\}^k$, fixing the Northwest and South boundaries to be strings of length $2n$ and $n$ respectively, one has
		
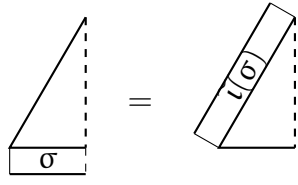
\[
		\begin{tikzpicture}[baseline=5mm,math mode,scale=0.5]
		\draw[thick] (2,0) -- (0,0) -- (2,3.46);
		\draw[dashed,thick] (2,3.46) -- (2,0);
		\draw[thick] (0,-0.7) -- (2,-0.7); 
		\draw (0,0) -- (0,-0.7); 
		\draw[dashed] (2,0) -- (2,-0.7);
		\node at (1,-0.35) {\sigma};
		\end{tikzpicture}
		\quad=\quad
		\begin{tikzpicture}[baseline=5mm,math mode,scale=0.5]
		\draw[thick] (2,0) -- (0,0) -- (2,3.46);
		\draw[dashed,thick] (2,3.46) -- (2,0);
		\draw[thick] (150:0.75) -- ++(60:4);
		\draw (0,0) -- (150:0.75) (60:4) -- ++(150:0.75);
		\path (60:2) ++(150:0.35) node[rotate=60] {\wt\iota(\sigma)};
		\end{tikzpicture}
		\]
		
	\end{prop}
	
	\begin{proof} It suffices to consider $\widetilde{\sigma}$ 
		(from Proposition \ref{prop:AJSB})
		a simple reflection. For the purposes of illustration, we set $n=4$ and 
		demonstrate the equality in the case of an $s_i$ where $i<n$, as well as 
		for $s_n$.
		\def\myscale{0.65}
		\[
		\begin{tikzpicture}[baseline=1.5cm,scale=\myscale,every node/.style={inner sep=.1em}]
		\foreach \i in {1,...,4} \draw[dg,arrow=1-0.4/\i] (\i+0.5,0.5) -- node[left=-1mm,pos=1-0.3/\i,black] {$\sss u_{\i}$}++(-\i/2,\i/2);
		\foreach \i in {1,...,4} \draw[dg,arrow=1-0.4/\i] (5,5-\i) -- node[left=-1mm,pos=1-0.3/\i,black] {$\sss -u_{\i}$}++(-\i/2,\i/2);
		\foreach \i in {1,...,4} \draw[dr,arrow=0.6*\i/(\i+0.2)] (\i+0.5,0.5) -- (5,5-\i);
		\draw[rounded corners,db,invarrow=0.3] (2.5,0.5) -- (2.5,0.25) -- (3.5,-0.25) -- node[right,pos=0.5,black] {$\sss u_2$} (3.5,-0.5);
		\draw[rounded corners,db,invarrow=0.3] (3.5,0.5) -- (3.5,0.25) -- (2.5,-0.25) -- node[left,pos=0.5,black] {$\sss u_3$}(2.5,-0.5);
		\draw[db,invarrow=0.3,invarrow=0.8] (1.5,0.5) -- node[left,pos=0.85,black] {$\sss u_1$}(1.5,-0.5) (4.5,0.5) -- node[right,pos=0.85,black] {$\sss u_4$}(4.5,-0.5);
		\draw[dashed,d] (5,-0.5) -- (5,4.75);
		\end{tikzpicture}
		\stackrel{(\ref{eqn:triv})}{=}\hspace{-0.5cm}
		\begin{tikzpicture}[baseline=1.5cm,scale=\myscale,every node/.style={inner sep=.1em}]
		\foreach \i in {1,4} \draw[dg,arrow=1-0.4/\i] (\i+0.5,0.5) -- node[left=-1mm,pos=1-0.3/\i,black] {$\sss u_{\i}$}++(-\i/2,\i/2);
		\draw[dg,rounded corners,arrow=0.85] (3.5,0.5) -- (2.5,0.75)-- (2,1) -- node[left=-1mm,pos=0.6,black]{$\sss u_2$}(1.5,1.5);
		\draw[dg,rounded corners,arrow=0.75] (2.5,0.5)-- (2.25,1) -- node[left=-1mm,pos=0.6,black]{$\sss u_3$}(2,2) ;
		\foreach \i in {1,...,4} \draw[dg,arrow=1-0.4/\i] (5,5-\i) -- node[left=-1mm,pos=1-0.3/\i,black] {$\sss -u_{\i}$}++(-\i/2,\i/2);
		\draw[dr,rounded corners,invarrow=0.4] (5,2) -- (4,1) -- (2.5,0.5);
		\draw[dr,rounded corners,invarrow=0.6] (5,3) -- (4,2) -- (3.5,1.25) -- (3.5,0.5);
		\foreach \i in {1,4} \draw[dr,arrow=0.6*\i/(\i+0.2)] (\i+0.5,0.5) -- (5,5-\i);
		\foreach \i in {1,4} \draw[db,invarrow=0.5] (\i+0.5,0.5) -- node[left,pos=0.85,black] {$\sss u_{\i}$}(\i+0.5,-0.5);
		\draw[db,invarrow=0.5] (2.5,0.5) -- node[left,pos=0.85,black] {$\sss u_3$}(2.5,-0.5);
		\draw[db,invarrow=0.5] (3.5,0.5) -- node[left,pos=0.85,black] {$\sss u_2$}(3.5,-0.5);
		\draw[dashed,d] (5,-0.5) -- (5,4.75);
		\end{tikzpicture} 
		\stackrel{(\ref{eqn:YBRG})}{=}\hspace{-0.5cm}
		\begin{tikzpicture}[baseline=1.5cm,scale=\myscale,every node/.style={inner sep=.1em}]
		\foreach \i in {1,4} \draw[dg,arrow=1-0.4/\i] (\i+0.5,0.5) -- node[left=-1mm,pos=1-0.3/\i,black] {$\sss u_{\i}$}++(-\i/2,\i/2);
		\draw[dg,rounded corners,arrow=0.9] (3.5,0.5) -- (2.5,1.4) -- node[below=0.5mm,pos=0.9,black]{$\sss u_2$}(1.5,1.5);
		\draw[dg,rounded corners,arrow=0.9] (2.5,0.5)-- (2,1.25) -- node[right=-1mm,pos=0.6,black]{$\sss u_3$}(2,2) ;
		\foreach \i in {1,...,4} \draw[dg,arrow=1-0.4/\i] (5,5-\i) -- node[left=-1mm,pos=1-0.3/\i,black] {$\sss -u_{\i}$}++(-\i/2,\i/2);
		\draw[dr,rounded corners,invarrow=0.475] (5,2) -- (4.5,1.8) -- (4,2) -- (2.5,0.5);
		\draw[dr,rounded corners,invarrow=0.675] (5,3) -- (4.5,2.5) -- (4.5,1.5) -- (4,1) -- (3.5,0.5);
		\foreach \i in {1,4} \draw[dr,arrow=0.6*\i/(\i+0.2)] (\i+0.5,0.5) -- (5,5-\i);
		\foreach \i in {1,4} \draw[db,invarrow=0.5] (\i+0.5,0.5) -- node[left,pos=0.85,black] {$\sss u_{\i}$}(\i+0.5,-0.5);
		\draw[db,invarrow=0.5] (2.5,0.5) -- node[left,pos=0.85,black] {$\sss u_3$}(2.5,-0.5);
		\draw[db,invarrow=0.5] (3.5,0.5) -- node[left,pos=0.85,black] {$\sss u_2$}(3.5,-0.5);
		\draw[dashed,d] (5,-0.5) -- (5,4.75);
		\end{tikzpicture}\stackrel{(\ref{eqn:reflection})}{=}\hspace{-0.5cm}
		\begin{tikzpicture}[baseline=1.5cm,scale=\myscale,every node/.style={inner sep=.1em}]
		\foreach \i in {1,4} \draw[dg,arrow=1-0.4/\i] (\i+0.5,0.5) -- node[left=-1mm,pos=1-0.3/\i,black] {$\sss u_{\i}$}++(-\i/2,\i/2);
		\draw[dg,rounded corners,arrow=0.9] (3.5,0.5) -- (2.5,1.4) -- node[below=0.5mm,pos=0.9,black]{$\sss u_2$}(1.5,1.5);
		\draw[dg,rounded corners,arrow=0.9] (2.5,0.5)-- (2,1.25) -- node[right=-1mm,pos=0.6,black]{$\sss u_3$}(2,2) ;
		\foreach \i in {1,4} \draw[dg,arrow=1-0.4/\i] (5,5-\i) -- node[left=-1mm,pos=1-0.3/\i,black] {$\sss -u_{\i}$}++(-\i/2,\i/2);
		\draw[dg,rounded corners, arrow=0.8] (5,3) -- (4.75,3.2) -- (4,3) -- node[left=-1mm,pos=0.5,black] {$\sss -u_3$}(3.5,3.5);
		\draw[dg,rounded corners, arrow=0.85] (5,2) -- (4.5,2.5) -- (4.25,2.7) -- (4.5,3.5) -- node[left=-1mm,pos=0.5,black] {$\sss -u_2$}(4,4);
		\foreach \i in {1,...,4} \draw[dr,arrow=0.6*\i/(\i+0.2)] (\i+0.5,0.5) -- (5,5-\i);
		\foreach \i in {1,4} \draw[db,invarrow=0.5] (\i+0.5,0.5) -- node[left,pos=0.85,black] {$\sss u_{\i}$}(\i+0.5,-0.5);
		\draw[db,invarrow=0.5] (2.5,0.5) -- node[left,pos=0.85,black] {$\sss u_3$}(2.5,-0.5);
		\draw[db,invarrow=0.5] (3.5,0.5) -- node[left,pos=0.85,black] {$\sss u_2$}(3.5,-0.5);
		\draw[dashed,d] (5,-0.5) -- (5,4.75);
		\end{tikzpicture}
		\stackrel{(\ref{eqn:YBRG})}{=}\hspace{-0.5cm}
		\begin{tikzpicture}[baseline=1.5cm,scale=\myscale,every node/.style={inner sep=.1em}]
		\foreach \i in {1,4} \draw[dg,arrow=1-0.4/\i] (\i+0.5,0.5) -- node[left=-1mm,pos=1-0.3/\i,black] {$\sss u_{\i}$}++(-\i/2,\i/2);
		\draw[dg,rounded corners,arrow=0.9] (3.5,0.5) -- (2.5,1.4) -- node[below=0.5mm,pos=0.9,black]{$\sss u_2$}(1.5,1.5);
		\draw[dg,rounded corners,arrow=0.9] (2.5,0.5)-- (2,1.25) -- node[right=-1mm,pos=0.6,black]{$\sss u_3$}(2,2) ;
		\foreach \i in {1,4} \draw[dg,arrow=1-0.4/\i] (5,5-\i) -- node[left=-1mm,pos=1-0.3/\i,black] {$\sss -u_{\i}$}++(-\i/2,\i/2);
		\draw[dg,rounded corners, arrow=0.6] (5,3) -- (4.5,3.5) -- (4,3.65) -- node[below=-1mm,pos=0.9,black] {$\sss -u_3$}(3.5,3.5);
		\draw[dg,rounded corners, arrow=0.7] (5,2) -- (4,3) -- (3.85,3.5) -- node[right=-1mm,pos=0.9,black] {$\sss -u_2$}(4,4);
		\foreach \i in {1,...,4} \draw[dr,arrow=0.6*\i/(\i+0.2)] (\i+0.5,0.5) -- (5,5-\i);
		\foreach \i in {1,4} \draw[db,invarrow=0.5] (\i+0.5,0.5) -- node[left,pos=0.85,black] {$\sss u_{\i}$}(\i+0.5,-0.5);
		\draw[db,invarrow=0.5] (2.5,0.5) -- node[left,pos=0.85,black] {$\sss u_3$}(2.5,-0.5);
		\draw[db,invarrow=0.5] (3.5,0.5) -- node[left,pos=0.85,black] {$\sss u_2$}(3.5,-0.5);
		\draw[dashed,d] (5,-0.5) -- (5,4.75);
		\end{tikzpicture}	
		\]
		\[
		\begin{tikzpicture}[baseline=1.5cm,scale=\myscale,every node/.style={inner sep=.1em}]
		\foreach \i in {1,...,4} \draw[dg,arrow=1-0.4/\i] (\i+0.5,0.5) -- node[left=-1mm,pos=1-0.3/\i,black] {$\sss u_{\i}$}++(-\i/2,\i/2);
		\foreach \i in {1,...,4} \draw[dg,arrow=1-0.4/\i] (5,5-\i) -- node[left=-1mm,pos=1-0.3/\i,black] {$\sss -u_{\i}$}++(-\i/2,\i/2);
		\foreach \i in {1,...,4} \draw[dr,arrow=0.6*\i/(\i+0.2)] (\i+0.5,0.5) -- (5,5-\i);
		\foreach \i in {1,...,3} \draw[db,invarrow=0.5] (\i+0.5,0.5) -- node[left,pos=0.85,black] {$\sss u_{\i}$}(\i+0.5,-0.5);
		\draw[db,rounded corners,invarrow=0.6] (4.5,0.5) -- (4.5,0.3)-- node[left=-1mm,pos=0.1,black] {$\sss u_4$} ++ (0.5,-0.3);
		\draw[db,invarrow=0.6] (5,0)-- node[left=0mm,pos=0.55,black] {$\sss -u_4$} (4.5,-0.5);
		\draw[dashed,d] (5,-0.5) -- (5,4.75);
		\end{tikzpicture}
		\stackrel{(\ref{eqn:fusion})}{=}\hspace{-0.5cm}
		\begin{tikzpicture}[baseline=1.5cm,scale=\myscale,every node/.style={inner sep=.1em}]
		\foreach \i in {1,...,3} \draw[dg,arrow=1-0.4/\i] (\i+0.5,0.5) -- node[left=-1mm,pos=1-0.3/\i,black] {$\sss u_{\i}$}++(-\i/2,\i/2);
		\foreach \i in {1,...,3} \draw[dg,arrow=1-0.4/\i] (5,5-\i) -- node[left=-1mm,pos=1-0.3/\i,black] {$\sss -u_{\i}$}++(-\i/2,\i/2);
		\draw[dg,rounded corners, arrow=0.9] (4.5,0.5) -- (4.25,0.7) -- (4.5,1.2) -- (4.5,1.5)-- node[left,pos=0.8,black] {$\sss -u_4$}(3,3);
		\draw[dg,rounded corners, arrow=0.9] (5,1) -- (4.75,1.2) -- (4.25,1) -- (4,1) --  node[left,pos=0.8,black] {$\sss u_4$}(2.5,2.5);
		\foreach \i in {1,...,4} \draw[dr,arrow=0.6*\i/(\i+0.2)] (\i+0.5,0.5) -- (5,5-\i);
		\foreach \i in {1,...,3} \draw[db,invarrow=0.5] (\i+0.5,0.5) -- node[left,pos=0.85,black] {$\sss u_{\i}$}(\i+0.5,-0.5);
		\draw[db,invarrow=0.5] (4.5,0.5) -- node[left,pos=0.85,black] {$\sss -u_4$}(4.5,-0.5);
		\draw[dashed,d] (5,-0.5) -- (5,4.75);
		\end{tikzpicture}
		\stackrel{(\ref{eqn:YBRG})}{=}\hspace{-0.5cm}
		\begin{tikzpicture}[baseline=1.5cm,scale=\myscale,every node/.style={inner sep=.1em}]
		\foreach \i in {1,...,3} \draw[dg,arrow=1-0.4/\i] (\i+0.5,0.5) -- node[left=-1mm,pos=1-0.3/\i,black] {$\sss u_{\i}$}++(-\i/2,\i/2);
		\foreach \i in {1,...,3} \draw[dg,arrow=1-0.4/\i] (5,5-\i) -- node[left=-1mm,pos=1-0.3/\i,black] {$\sss -u_{\i}$}++(-\i/2,\i/2);
		\draw[dg,rounded corners, arrow=0.9] (4.5,0.5) -- (3,2) -- (3,2.25) -- (3.1,2.75)-- node[right=-1mm,pos=0.8,black] {$\sss -u_4$}(3,3);
		\draw[dg,rounded corners, arrow=0.9] (5,1) -- (3.5,2.5) -- (3.2,2.5) -- (2.65,2.4) --  node[below,pos=0.8,black] {$\sss u_4$}(2.5,2.5);
		\foreach \i in {1,...,4} \draw[dr,arrow=0.6*\i/(\i+0.2)] (\i+0.5,0.5) -- (5,5-\i);
		\foreach \i in {1,...,3} \draw[db,invarrow=0.5] (\i+0.5,0.5) -- node[left,pos=0.85,black] {$\sss u_{\i}$}(\i+0.5,-0.5);
		\draw[db,invarrow=0.5] (4.5,0.5) -- node[left,pos=0.85,black] {$\sss -u_4$}(4.5,-0.5);
		\draw[dashed,d] (5,-0.5) -- (5,4.75);
		\end{tikzpicture}
		\]
	\end{proof}

	\begin{lem}\label{lem:dab} \begin{enumerate}[a),leftmargin=0.2in]
			\item \cite[Proposition 4]{KZJ} Type $A$. Let $\sigma \in 0^k1^{2n-k}$ and $\la$ be a string of length $2n$: \\ If \;
			$\begin{tikzpicture}[baseline=-5mm, math mode]
			\draw[thick] (0,-0.6) -- node[pos=0.5] {\omega_{Gr}} (2,-0.6); 
			\draw (0,0) -- (0,-0.6); \draw[dashed] (2,0) -- (2,-0.6); 
			\draw (0,0) -- node[pos=0.5] {\la} (2,0); \node at (1,-0.3) {\sigma};
			\end{tikzpicture}\neq 0$
			for $\omega_{Gr}$ as in Proposition \ref{prop:ajsbwiring}, then $\la$ consists only of $0$s and $1$s (no $10$s).
			\item Type $C$. Let $\sigma \in (10)^{n-k}\{0,1\}^k$ and $\la$ be a string of length $n$: If \;
			$\begin{tikzpicture}[baseline=-5mm, math mode]
			\draw[thick] (0,-0.6) -- node[pos=0.5] {\omega_{SpGr}} (2,-0.6); 
			\draw (0,0) -- (0,-0.6); \draw[dashed] (2,0) -- (2,-0.6); 
			\draw (0,0) -- node[pos=0.5] {\la} (2,0); \node at (1,-0.3) {\sigma};
			\end{tikzpicture}\neq 0$
			for $\omega_{SpGr}$ as in Proposition \ref{prop:ajsbwiring}, then $\la$ has the same number of $10$s as $\omega_{SpGr}$.
		\end{enumerate}
	\end{lem}
	
	\begin{proof} 
		To prove part b), recall that
		$\begin{tikzpicture}[baseline=-6mm, math mode] \draw[thick] (0,-0.6)
		-- node[pos=0.5] {\omega_{SpGr}} (2,-0.6); \draw (0,0) -- (0,-0.6);
		\draw (2,0)[dashed] -- (2,-0.6); \draw (0,0) -- node[pos=0.5]
		{\la} (2,0); \node at (1,-0.3) {\sigma};
		\end{tikzpicture}$
		is the $(\lambda,\omega_{SpGr})$ matrix entry for the composition of
		$R_{BB}$ and $K_B$ maps. From Definition \ref{def:RKU}, we see that both of these maps preserve the number of $10$s in a string, 
		hence so will compositions of these maps.
	\end{proof}

	\begin{proof}[Proof of Theorem \ref{thm:main}] In $H^{\ast}_T(\text{pt})$, we have the following equality
		\[
		\begin{tikzpicture}[baseline=5mm,math mode,scale=0.75]
		\draw[thick] (150:0.8) -- ++(60:4) node[pos=0.5] {\la\ };
		\draw[thick] (150:0) -- ++(60:4) node[right=-2.5mm,pos=0.5] {\omega_{Gr} \ };
		\draw (0,0) -- (150:0.8) (60:4) -- ++(150:0.8);
		\path (60:2) ++(150:0.4) node[rotate=60] {\wt\iota(\sigma)};
		\end{tikzpicture}\hspace{-0.1cm}
		\stackrel{(\text{L}\ref{lem:Kdelta})}{=}
		\sum_{\mu}
		\begin{tikzpicture}[baseline=5mm,math mode,scale=0.75]
		\draw[thick] (2,0) -- node {\omega_{SpGr}} (0,0) -- (2,3.46);
		\draw[dashed,thick] (2,3.46) -- (2,0);
		\draw[thick] (150:0.8) -- ++(60:4) node[pos=0.5] {\la\ };
		\draw[thick] (150:0) -- ++(60:4) node[pos=0.5] {\mu};
		\draw (0,0) -- (150:0.8) (60:4) -- ++(150:0.8);
		\path (60:2) ++(150:0.4) node[rotate=60] {\wt\iota(\sigma)};
		\end{tikzpicture}
		\stackrel{(\text{L}\ref{lem:dab}a)}{=}
		\begin{tikzpicture}[baseline=5mm,math mode,scale=0.75]
		\draw[thick] (2,0) -- node {\omega_{SpGr}} (0,0) -- (2,3.46);
		\draw[dashed,thick] (2,3.46) -- (2,0);
		\draw[thick] (150:0.8) -- ++(60:4) node[pos=0.5] {\la\ };
		\draw (0,0) -- (150:0.8) (60:4) -- ++(150:0.8);
		\path (60:2) ++(150:0.4) node[rotate=60] {\wt\iota(\sigma)};
		\end{tikzpicture}\;
		\stackrel{(\text{P}\ref{prop:key})}{=}
		\begin{tikzpicture}[baseline=5mm,math mode,scale=0.75]
		\draw[thick] (2,0)  -- (0,0) -- node {\lambda\ } (2,3.46);
		\draw[dashed,thick] (2,0) -- (2,3.46);
		\draw[thick] (0,-0.7) -- node {\omega_{SpGr}} (2,-0.7); 
		\draw (0,0) -- (0,-0.7); 
		\draw[dashed] (2,0) -- (2,-0.7);
		\node at (1,-0.35) {\sigma};
		\end{tikzpicture}
		\stackrel{(\text{L}\ref{lem:dab}b)}{=}
		\sum_{\nu}
		\begin{tikzpicture}[baseline=5mm,math mode,scale=0.75]
		\draw[thick] (2,0)  -- node[pos=0.5] {\nu} (0,0) -- node {\lambda\ } (2,3.46);
		\draw[dashed,thick] (2,0) -- (2,3.46);
		\draw[thick] (0,-0.7) -- node {\omega_{SpGr}} (2,-0.7); 
		\draw (0,0) -- (0,-0.7); 
		\draw[dashed] (2,0) -- (2,-0.7);
		\node at (1,-0.35) {\sigma};
		\end{tikzpicture}
		\]
		The left side corresponds to $\iota^{\ast}(S_{\la})|_{\sigma}$ by Proposition \ref{prop:ajsbwiring}. In the second and fourth equality, the strings $\mu$ and $\nu$ have content $0^k1^{2n-k}$ and $(10)^{n-k}\left\{0,1\right\}^k$ respectively, and all other terms of the sum vanish. 
	\end{proof}

	\noindent{\bf Acknowledgements.}
	We thank \.Izzet Co\c skun and Alex Yong for references,
	and Michael Wheeler for discussions about $P$- and $Q$-Schur functions.

\end{document}